\documentclass{amsart}

\usepackage{amsmath,amssymb,amsthm}
\usepackage[all]{xy}

\DeclareMathOperator{\Diff}{Diff}
\DeclareMathOperator{\Aut}{Aut}
\DeclareMathOperator{\Lie}{Lie}

\DeclareMathOperator{\SL}{SL}
\DeclareMathOperator{\GCD}{GCD}

\newtheorem{thm}{Theorem}
\newtheorem{lem}{Lemma}
\newtheorem{prop}{Proposition}
\newtheorem{cor}{Corollary}

\begin{document}
\title[Rigidity of the \'{A}lvarez class]{Rigidity of the \'{A}lvarez class of Riemannian foliations with nilpotent structure Lie algebras}
\author{Hiraku Nozawa}
\address{Graduate School of Mathematical Sciences, University of Tokyo, 3-8-1 Komaba, Meguro, Tokyo, 153-8914, Japan}
\email{nozawahiraku@06.alumni.u-tokyo.ac.jp}

\begin{abstract}
We show that if the structure algebra of a Riemannian foliation $\mathcal{F}$ on a closed manifold $M$ is nilpotent, the integral of the \'{A}lvarez class of $(M,\mathcal{F})$ along every closed path is the exponential of an algebraic number. As a corollary, we prove that the \'{A}lvarez class and geometrically tautness of Riemannian foliations on a closed manifold $M$ are invariant under deformation, if the fundamental group of $M$ has polynomial growth.
\end{abstract}
\maketitle

\section{Introduction}

A Riemannian foliation $\mathcal{F}$ on a closed manifold $M$ is geometrically taut if there exists a bundle-like metric $g$ on $M$ such that every leaf of $\mathcal{F}$ is a minimal submanifold of $(M,g)$. Geometrically tautness of Riemannian foliations is a purely differential geometric property, but it is known that it has remarkable relations with the cohomological properties of $(M,\mathcal{F})$. For examples, Masa~\cite{Mas} characterized the geometrically tautness of $(M,\mathcal{F})$ by the nontriviality of the top degree part of the basic cohomology of $(M,\mathcal{F})$ and \'{A}lvarez L\'{o}pez~\cite{Alv} defined a cohomology class $[\kappa_b]$ of degree $1$ for $(M,\mathcal{F})$ which vanishes if and only if $(M,\mathcal{F})$ is geometrically taut. We call $[\kappa_b]$ the \'{A}lvarez class of $(M,\mathcal{F})$. In this paper, we show that \'{A}lvarez classes of Riemannian foliations have a rigidity property, if the structure Lie algebra is nilpotent. As a corollary of the rigidity of \'{A}lvarez classes, we obtain the invariance of geometrically tautness of Riemannian foliations under deformation, if the fundamental group of the ambient manifold has polynomial growth.

The main result in this paper is the following theorem: 

\begin{thm}\label{expalg}
 Let $M$ a closed manifold and $\mathcal{F}$ be a Riemannian foliation on $M$ with nilpotent structure Lie algebra. Then $e^{\int_{\gamma} [\kappa_b]}$ is an algebraic number for every $\gamma$ in $\pi_1(M)$ where $[\kappa_b]$ is the \'{A}lvarez class of $(M,\mathcal{F})$.
\end{thm}
\noindent
Theorem~\ref{expalg} is shown by computation of \'{A}lvarez classes in terms of the  holonomy of the basic fibration in Section 2 and application of Mal'cev theory in Section 3.

We state two corollaries of Theorem~\ref{expalg}. Let $M$ be a closed manifold whose fundamental group has polynomial growth. Then the structure Lie algebra of every Riemannian foliation on $M$ is nilpotent according to Carri\`{e}re~\cite{Car}. By Theorem~\ref{expalg}, the \'{A}lvarez classes of Riemannian foliations on $M$ are contained in a countable subset of $H^1(M;\mathbb{R})$ which is independent of foliations. On the other hand, if we have a smooth family $\{\mathcal{F}^{t}\}_{t \in [0,1]}$ of Riemannian foliation on $M$, their \'{A}lvarez classes varies continuously in $H^1(M;\mathbb{R})$ as shown in~\cite{Noz}. Hence we have the following corollary:

\begin{cor}\label{alvinv}
Let $M$ be a closed manifold whose fundamental group has polynomial growth and $\{\mathcal{F}^{t}\}_{t \in [0,1]}$ be a smooth family of Riemannian foliation on $M$ over $[0,1]$. Then we have $[\kappa_{b}^{t}]=[\kappa_{b}^{t'}]$ in $H^1(M;\mathbb{R})$ for any $t$ and $t'$ in $[0,1]$ where $[\kappa_{b}^{t}]$ is the \'{A}lvarez class of $(M,\mathcal{F}^{t})$. 
\end{cor}
Since $(M,\mathcal{F}^{t})$ is geometrically taut if and only if the \'{A}lvarez class of $(M,\mathcal{F}^{t})$ vanishes according to \'{A}lvarez L\'{o}pez~\cite{Alv}, we have the following corollary of Corollary~\ref{alvinv}:

\begin{cor}\label{definv}
Let $M$ be a closed manifold whose fundamental group has polynomial growth and $\{\mathcal{F}^{t}\}_{t \in [0,1]}$ be a smooth family of Riemannian foliation on $M$. Then one of the following is true:
\begin{enumerate}
\item $(M,\mathcal{F}^{t})$ is geometrically taut for every $t$ in $[0,1]$.
\item $(M,\mathcal{F}^{t})$ is not geometrically taut for every $t$ in $[0,1]$.
\end{enumerate}
\end{cor}

If the dimension of $\mathcal{F}$ is $1$, the structure Lie algebra of $\mathcal{F}$ is always abelian according to Caron-Carri\`{e}re~\cite{CaCa} or Carri\`{e}re~\cite{Car}. Hence the conclusion of Corollary~\ref{alvinv} and Corollary~\ref{definv} follow from Theorem~\ref{expalg} without the assumption on the growth of the fundamental group of $M$ for the cases where $\mathcal{F}$ is $1$-dimensional.

Note that we can show Theorem~\ref{expalg} by more direct computation in~\cite{Noz} if the dimension of $\mathcal{F}$ is $1$, using a theorem of Caron-Carri\`{e}re~\cite{CaCa} which claims $1$-dimensional Lie foliations with dense leaves are linear flows on tori with irrational slopes and computation of the mapping class group of the group of diffeomorphisms preserving a linear foliation with dense leaves on tori by Molino and Sergiescu~\cite{MoSe}.

The author expresses his gratitude to Jes\'{u}s Antonio \'{A}lvarez L\'{o}pez, Jos\'{e} Royo Prieto Ignacio and Steven Hurder. The conversation with \'{A}lvarez L\'{o}pez and Royo Prieto on \'{A}lvarez classes gave the author a definite clue to carry out the computation of \'{A}lvarez classes in Section 2. The idea on conditions for the growth of the fundamental group of the ambient manifold came from the conversation with Hurder.

\section{Computation of \'{A}lvarez classes in terms of the basic fibration}

We prepare the notation to state the main result in this section. Let $(M,\mathcal{F})$ be a closed manifold with an orientable transversely parallelizable foliation. Let $\pi \colon M \longrightarrow W$ be the basic fibration of $(M,\mathcal{F})$. Fix a point $x$ on $M$. We denote a fiber of $\pi$ which contains $x$ by $N$ and the restriction of $\mathcal{F}$ to $N$ by $\mathcal{G}$. Then $(N,\mathcal{G})$ is a Lie foliation by Molino's structure theorem. We denote the dimension and codimension of $(N,\mathcal{G})$ by $k$ and $l$ respectively. Let $\mathfrak{g}$ be the structure algebra of $(N,\mathcal{G})$, which has dimension $l$. Throughout Section 2, we assume the unimodularity of $\mathfrak{g}$, that is,
\begin{equation}\label{unimod}
H^{l}(\mathfrak{g}) \cong \mathbb{R}.
\end{equation}
Note that assumption~\eqref{unimod} is satisfied if $\mathfrak{g}$ is nilpotent and may not be satisfied if $\mathfrak{g}$ is solvable.

For the basic definition of the spectral sequences of foliated manifolds, we refer to a paper by Kamber and Tondeur~\cite{KaTo2}. Let $E_{2}^{0,k}$ be the $(0,k)$-th $E_2$ term of the spectral sequence of $(N,\mathcal{G})$. Then the dimension of $E_{2}^{0,k}$ is $1$ by the assumption~\eqref{unimod}, since we have 
\begin{equation}
H^{l}(\mathfrak{g}) \cong H_{B}^{l}(N/\mathcal{G}) \cong E_{2}^{0,k}
\end{equation}
where the first isomorphism follows from the denseness of the leaves of $(N,\mathcal{G})$ and the second isomorphism follows from the duality theorem of Masa in~\cite{Mas}. Hence $\Aut (E_{2}^{0,k})$ is canonically identified with $\mathbb{R}-\{0\}$. We denote the orientation preserving automorphism group of $E_{2}^{0,k}$ by $\Aut_{+} (E_{2}^{0,k})$, which is identified with the set of positive real numbers $\mathbb{R}_{>0}$.

Let $\Diff(N,\mathcal{G})$ be the group of diffeomorphisms on $N$ which map each leaf of $\mathcal{G}$ to a leaf of $\mathcal{G}$ and denote its mapping class group in the $C^{\infty}$ topology by $\pi_{0}(\Diff(N,\mathcal{G}))$. We have a canonical action $\Phi \colon \pi_{0}(\Diff(N,\mathcal{G})) \longrightarrow \Aut(E_{2}^{0,k})$ defined in the following way: Let $H^{k}(\mathcal{G})$ be the $k$-th leafwise cohomology group of $(N,\mathcal{G})$, which is identified with $E_{1}^{0,k}$ in the spectral sequence of $(N,\mathcal{G})$. $E_{2}^{0,k}$ is identified with the kernel of $d_{1,0} \colon H^{k}(\mathcal{G}) \longrightarrow E_{1}^{1,k}$ where $d_{1,0}$ is the map induced on $E_1$ terms from the composition of the de Rham differential and the projection $C^{\infty}(\wedge^{k+1} T^{*}N) \longrightarrow C^{\infty}((T\mathcal{G}^{\perp})^{*} \otimes \wedge^{k} T^{*}\mathcal{G})$ determined by a Riemann metric $g$ of $(N,\mathcal{G})$ (See~\cite{KaTo2}.). $\Diff(N,\mathcal{G})$ acts to $E_{2}^{0,k}$, since $\Diff(N,\mathcal{G})$ acts to the leafwise cohomology group by pulling back leafwise volume forms and this action preserves $E_{2}^{0,k}$. This action descends to an action of $\pi_{0}(\Diff(N,\mathcal{G}))$, since the action of the identity component of $\Diff(N,\mathcal{G})$ to the leafwise cohomology group is trivial according to the integrable homotopy invariance of leafwise cohomology shown by El Kacimi Alaoui~\cite{ElK}.

We show the following proposition which computes the period of the \'{A}lvarez class $[\kappa_b]$ of $(M,\mathcal{F})$ in terms of the holonomy of the basic fibration:

\begin{prop}\label{holonomyformula}
The diagram
\begin{equation}\label{diag1}
\xymatrix{ \pi_1(M,x) \ar[rr]^{\int [\kappa_b]} \ar[d]_{\pi_{*}} &  & \mathbb{R} \\
           \pi_1(W,\pi(x)) \ar[r]^{hol_{\pi}} & \pi_{0}(\Diff(N,\mathcal{G})) \ar[r]^{\Phi} & \Aut_{+}(E_{2}^{0,k}) \ar[u]_{\log}}
\end{equation}
is commutative where $\int [\kappa_b]$ is the period map of $[\kappa_b]$, $hol_{\pi}$ is the holonomy map of the basic fibration $\pi \colon M \longrightarrow W$ , $\log$ is defined through the identification of $\Aut(E_{2}^{0,k})$ with $\mathbb{R}_{>0}$ and $\Phi$ is the canonical action described above.
\end{prop}

\begin{proof}
To show the commutativity of the diagram~\eqref{diag1} for an element $[\gamma]$ of $\pi_1(M,x)$ which is represented by a smooth path $\gamma$, it suffices to show the case of $W=S^1$ pulling back the fibration $\pi$ by $\gamma$. By the assumption~\eqref{unimod}, the \'{A}lvarez class of $(N,\mathcal{G})$ cannot be nontrivial. Hence the restriction of \'{A}lvarez class of $(M,\mathcal{F})$ to a fiber of $\pi$ is zero. We will compute the integration of \'{A}lvarez class of $(M,\mathcal{F})$ along a path which gives a generator of $\pi_1(S^1,\pi(x))$. We denote the holonomy of the $(N,\mathcal{G})$-bundle $\pi$ over $S^1$ along a path which gives a generator of $\pi_1(S^1,\pi(x))$ by $f$ and its action on $E_{2}^{0,k}$ by $f^*$. We write $\overline{\mathcal{F}}$ for the foliation defined by the fibers of $\pi$.

We fix a bundle-like metric $g'$ on $(M,\mathcal{F})$. Let $E$ be the vector bundle of rank $1$ over $S^1$ whose fiber $E_t$ over $t$ in $S^1$ is the $(0,k)$-th $E_2$ term of the spectral sequence of $(\pi^{-1}(t),\mathcal{F}|_{\pi^{-1}(t)})$. We define an affine connection $\nabla$ on $E$ by
\begin{equation} 
\nabla s = \int_{\pi} (d_{1,0} s) 
\end{equation}
for $s$ in $C^{\infty}(E)$ where $d_{1,0}$ is the map induced on $E_1$ terms from the composition of the de Rham differential and the projection $C^{\infty}(\wedge^{k+1} T^{*}M) \longrightarrow C^{\infty}((T\mathcal{F}^{\perp})^{*} \otimes \wedge^{k} T^{*}\mathcal{F})$ determined by a Riemann metric $g$ of $(M,\mathcal{F})$, and $\int_{\pi}$ is the integration on fibers on $\pi$ with respect to the first component of $C^{\infty}(T^{*}M) \otimes C^{\infty}(E)$ using the fiberwise volume form $vol_{\pi}$ of $\pi$ determined by $g'$ which is defined by $\int_{\pi} (\alpha \otimes h[\chi]) = \big(\int_{\pi} h\alpha \wedge vol_{\pi} \big) \otimes [\chi]$ for $\alpha$ in $C^{\infty}(T^{*}M)$ and $h$ in $C^{\infty}(M)$. The Rummler's formula
\begin{equation}\label{Rum}
d_{1,0} \chi = -\kappa \wedge \chi
\end{equation}
in~\cite{Rum} implies $\nabla$ is an connection on $E$. Note that $d_{1,0}$ coincides with the differential on $E_1$ terms of the spectral sequence of $(M,\mathcal{F})$ which is determined only by $\mathcal{F}$ and hence $\nabla$ is independent of the metric $g'$.

The connection $\nabla$ is flat, since every connection on a vector bundle over $S^1$ is flat. The holonomy of $(E,\nabla)$ which corresponds to a generator of $\pi_1(S^1,\pi(x))$ is shown to be equal to $f^{*}$ in the following way: We pull back the $(N,\mathcal{G})$ bundle $\pi$ by the canonical map $\iota \colon [0,1] \longrightarrow [0,1]/ \{0\} \sim \{1\}=S^1$ and denote the total space of $\iota^{*}\pi$ by $(M',\mathcal{F}')$. Fix a trivialization $(M',\mathcal{F}') \cong (N,\mathcal{G}) \times [0,1]$ as a $(N,\mathcal{G})$ bundle, then we have an induced trivialization of $\iota^{*}E$. Since $\iota^{*}\nabla$ is independent of the metric, we can assume that $\iota^{*}\nabla$ is defined by a product metric. Then the parallel section of $(\iota^{*}E,\iota^{*}\nabla)$ is the constant sections with respect to the trivialization. Since $(E,\nabla)$ is obtained by identifying the boundaries of $(\iota^{*}E,\iota^{*}\nabla)$ by $f^*$, the holonomy of $(E,\nabla)$ is equal to $f^{*}$.

We will show that the holonomy of $(E,\nabla)$ which corresponds to a generator of $\pi_1(S^1,\pi(x))$ is equal to $e^{\int_{S^1} \kappa_b}$ where the holonomy of $(E,\nabla)$ is regarded as a real number. For this purpose, we construct a bundle-like metric $g$ on $(M,\mathcal{F})$ such that each leaf of $\overline{\mathcal{F}}$ is a minimal manifold of $(M,g)$ and the leafwise volume form $\chi^{t}$ of $(\pi^{-1}(t),\mathcal{F}|_{\pi^{-1}(t)})$ determined by $g$ satisfies $d_{1,0}^{t} \chi^{t}=0$ where $d_{1,0}^{t}$ is the transverse component of de Rham differential of $({\pi^{-1}(t)},\mathcal{F}|_{{\pi^{-1}(t)}})$ defined in the same way as $d_{1,0}$ for $(N,\mathcal{G})$. By the duality theorem of Masa in~\cite{Mas} and the assumption~\eqref{unimod}, we have a bundle-like metric $g_0$ on $(N,\mathcal{G})$ such that each leaf of $\mathcal{G}$ is a minimal submanifold of $(N,g_0)$. Note that the leafwise cohomology class $[\chi_0]$ of the characteristic form $\chi_0$ of $(N,\mathcal{G},g_0)$ is an eigenvector of $f^*$, since $[\chi_0]$ is a generator of $E_{2}^{0,k}$ and $f^*$ preserves $E_{2}^{0,k}$. Hence we can write $f^*[\chi_0]=c[\chi_0]$ for some real number $c$. By the Moser's argument in~\cite{Ghy}, we can isotope $f$ to $f_1$ in $\Diff(N,\mathcal{G})$ so that $f^{*}_{1}\chi_0=c\chi_0$. Let $\omega_0$ be a basic transverse volume form of $(N,\mathcal{G})$. Then we have $f^{*}_{1}\omega_0=b\omega_0$ for some real number $b$, the leaves of $(N,\mathcal{G})$ are dense. Since $f_{1}$ induces the identity map on $H^{k+l}(N ; \mathbb{R})$ and a pairing 
\begin{equation}\label{pair}
E_{2}^{0,k} \times E_{2}^{l,0} \longrightarrow E_{2}^{l,k}=H^{k+l}(N ; \mathbb{R})
\end{equation}
is natural with respect to $\Diff(N,\mathcal{G})$, we have $d=\frac{1}{c}$. Hence we have $f^{*}_{1} (\chi_0 \wedge \omega_0)=\chi_0 \wedge \omega_0$. We define a bundle-like metric $g$ of $(M,\mathcal{F})$ by $g= \rho(t)g_{0} +(1-\rho(t))f^{*}_{1}g_{0} + dt \otimes dt$ where $\rho$ is a smooth function on $[0,1]$ which satisfies $\rho(t)=0$ near $0$ and $\rho(t)=1$ near $1$. We denote the characteristic forms of $(M,\mathcal{F},g)$ by $\chi$. Then we have $\chi = \rho(t)\chi_{0} +(1-\rho(t))f^{*}_{1}\chi_{0}$. By the Rummler formula, we have $d_{1,0}\chi_0=0$ and hence $d_{1,0}^{t}(\chi|_{\pi^{-1}(t)})=0$ for every $t$ in $S^1$. Note that each leaf of $\overline{\mathcal{F}}$ is a minimal submanifold of $(M,g)$, since $f_1$ preserves the volume form of $(N,g_0)$.

We calculate the holonomy of $(E,\nabla)$ using $g$. We denote the map which corresponds the leafwise cohomology class $[\chi|_{\pi^{-1}(t)}]$ to $t$ in $S^1$ by $[\chi]$. Then $[\chi]$ is a global section of $E$, since $d_{1,0}^{t} \chi|_{\pi^{-1}(t)}=0$ is satisfied by the construction of $\chi$. By the Rummler's formula~\eqref{Rum}, we have 
\begin{equation}\label{conn}
\nabla [\chi] = -\Big(\int_{\pi} \kappa \Big) \otimes [\chi]. 
\end{equation}
We show that $\int_{\pi} \kappa$ is a closed $1$-form on $S^1$ which satisfies $\kappa_b(\frac{\partial}{\partial t})=\int_{\pi} \kappa(\frac{\partial}{\partial t})$ on $S^1$. By the minimality of the fibers of $\pi$ with respect to $g$ and Rummler's formula for $(M,\overline{\mathcal{F}})$, $dvol_{\pi}$ has no component which has $m$-form tangent to the fibers of $\pi$ where $m$ is the dimension of the fibers of $\pi$. Hence we have 

\begin{equation}
\begin{array}{rl}
d\int_{\pi} \kappa & = d\int_{\pi} \kappa_b \\
& =\int_{\pi} \Big(d\kappa_b \wedge vol_{\pi} - \kappa_b \wedge dvol_{\pi} \Big) \\
& = \int_{\pi} \Big(d\kappa_b \wedge vol_{\pi} - \kappa_b \wedge dvol_{\pi} \Big) \\
& = \int_{\pi} \Big(d\kappa_b \wedge vol_{\pi}\Big) \\
& = \int_{\pi} d\kappa_b.
\end{array}
\end{equation}
Since $\kappa_b$ is closed by Corollary~3.5 of \'{A}lvarez L\'{o}pez~\cite{Alv}, we have $d\int_{\pi} \kappa=0$. $\kappa_b(\frac{\partial}{\partial t})=\int_{\pi} \kappa(\frac{\partial}{\partial t})$ is clear by the definition of $\kappa_b$.

Hence $\nabla$ is a flat connection defined by a closed form $\int_{\pi} \kappa$ and we can show the holonomy of $(E,\nabla)$ is equal to $e^{\int_{S^1} \int_{\pi} \kappa}=e^{\int_{S^1} \kappa_b}$ in a standard way.
\end{proof}

\section{Application of Mal'cev Theory}

Let $(M,\mathcal{F})$ be a closed manifold with a Riemannian foliation with nilpotent structure Lie algebra. Since the \'{A}lvarez class of $(M,\mathcal{F})$ is defined by the integration along fibers of the \'{A}lvarez class of $(M^{1},\mathcal{F}^{1})$ where $M^{1}$ is the transverse orthonormal frame bundle of $(M,\mathcal{F})$ and $\mathcal{F}^{1}$ is the lift of $\mathcal{F}$, which is transversely parallelizable, Theorem~\ref{expalg} is reduced to the case where $(M,\mathcal{F})$ is transversely parallelizable. Moreover we can assume the orientability of $\mathcal{F}$, since a square root of an algebraic number is algebraic. To show Theorem~\ref{expalg} in the cases where $(M,\mathcal{F})$ is orientable and transversely parallelizable, it suffices to show the following Proposition~by Proposition~\ref{holonomyformula}:  

\begin{prop}\label{algebraicaction}
Let $(N,\mathcal{G})$ be a closed manifold with a Lie foliation with nilpotent structure Lie algebra of which every leaf is dense in $N$. Then the image of $\Phi \colon \pi_{0}(\Diff(N,\mathcal{G})) \longrightarrow \Aut(E_{2}^{0,k})$ is contained in the set of algebraic numbers where $\Aut(E_{2}^{0,k})$ is canonically identified with $\mathbb{R}-\{0\}$.
\end{prop}

We show a lemma which reduces Proposition~\ref{algebraicaction} to a problem on nilpotent Lie groups and will apply Mal'cev theory to complete the proof of Proposition~\ref{algebraicaction}.

Let $N$ be a closed manifold and $\mathcal{G}$ be a Lie foliation on $N$ of dimension $k$ and codimension $l$ of which every leaf is dense in $N$. We denote the structure Lie algebra of $(N,\mathcal{G})$ by $\mathfrak{g}$. We fix a point $x$ on $N$. Then the holonomy homomorphism $hol \colon \pi_1(N,x) \longrightarrow G$ of $(N,\mathcal{G})$ is determined where $G$ is the simply connected structure Lie group determined by $\mathfrak{g}$.

 We denote the canonical action $\Diff(N,\mathcal{G}) \longrightarrow \Aut(E_{2}^{0,l})$ by $\Psi$ where $E_{2}^{l,0}$ is the $(0,l)$-th $E_2$ term of the spectral sequence of $(N,\mathcal{G})$. Note that $E_{2}^{l,0}$ is isomorphic to the $l$-th basic cohomology group of $(N,\mathcal{G})$~\cite{Hae} and hence isomorphic to $H^{l}(\mathfrak{g})$, since the leaves of $\mathcal{G}$ are dense in $N$.
\begin{lem}\label{liefol}
If we assume 
\begin{equation}\label{unimod2}
H^{l}(\mathfrak{g}) \cong \mathbb{R},
\end{equation}
then we have the following:
\begin{enumerate}
\item The diagram
\begin{equation}
\xymatrix{ \Diff(N,\mathcal{G}) \ar[r]^{\Phi} \ar[rd]_{\Psi} &  \Aut(E_{2}^{0,k}) \\
 & \Aut(E_{2}^{l,0}) \ar[u]_{i}}
\end{equation}
commutes where $i$ is the map which takes the inverse through identification of $\Aut(E_{2}^{0,k})$ and $\Aut(E_{2}^{l,0})$ with $\mathbb{R}-\{0\}$.
\item \label{2} We denote the image of $hol$ by $\Gamma$. A foliation preserving diffeomorphism $f \colon (N,\mathcal{G}) \longrightarrow (N,\mathcal{G})$ which fixes $x$ induces an automorphism of $G$ which preserves $\Gamma$.
\item We denote the group of diffeomorphisms which fix $x$ and preserve $\mathcal{G}$ by $\Diff(N,x,\mathcal{G})$, the group of automorphisms of $G$ which preserve $\Gamma$ by $\Aut(G, \Gamma)$ and the homomorphism $\Diff(N,x,\mathcal{G}) \longrightarrow \Aut(G)$ which is obtained by~\eqref{2} by $\iota$. Then the diagram 
\begin{equation}
\xymatrix{ \Diff(N,x,\mathcal{G}) \ar[r]^{\Phi} \ar[rd]_{\iota} & \Aut(E_{2}^{l,0}) \\
 & \Aut(G,\Gamma) \ar[u]_{A} }
\end{equation}
commutes where $A$ is the canonical action of the automorphism group of Lie group to the Lie algebra cohomology under the identification of $E_{2}^{l,0}$ with $H^{l}(\mathfrak{g})$.
\end{enumerate}
\end{lem}

\begin{proof}
We prove (1). Let $f$ be an element of $\Diff(N,\mathcal{G})$. Since the pairing~\eqref{pair} is natural with respect to $\Diff(N,\mathcal{G})$ and the action of $f$ to $H^{k+l}(N;\mathbb{R})$ is trivial, we have the conclusion.

We prove (2). Let $X$ be a basic transverse vector field on $(N,\mathcal{G})$ which we regard as an element of $\mathfrak{g}$. Then $f_{*}X$ is again an element of $\mathfrak{g}$ where $f_{*} \colon C^{\infty}(TN/T\mathcal{G}) \longrightarrow C^{\infty}(TN/T\mathcal{G})$ is the map induced by $f$, since every leafwise constant function of $(N,\mathcal{G})$ is constant. Hence we have an automorphism of $\mathfrak{g}$ induced by $f_{*}$, which we denote by $f_{*}$ again. Let $(\tilde{N},\tilde{\mathcal{G}})$ be the universal cover of $(N,\mathcal{G})$. We fix a point $\tilde{x}$ on $\tilde{N}$ such that $u(\tilde{x})=x$ where $u$ is the projection of the universal covering $\tilde{N} \longrightarrow N$. We can lift $f$ to $\tilde{f} \colon (\tilde{N},\tilde{\mathcal{G}}) \longrightarrow (\tilde{N},\tilde{\mathcal{G}})$ so that $\tilde{f}$ fixes $\tilde{x}$. $\tilde{f}$ induces a diffeomorphism $\overline{f}$ on $G=\tilde{N}/\tilde{\mathcal{G}}$. Let $g$ be an automorphism of $G$ which is induced by an element $(f_{*})^{-1}$ of $\Aut(\mathfrak{g})$. Then $\overline{f} \circ g = L_{h}$ for some $h$ in $G$, since $d(\overline{f} \circ g)$ preserves every left-invariant vector field on $G$. But since $\overline{f} \circ g$ fixes $e$, we have $\overline{f} \circ g = id_{G}$.

We prove the latter part. By the definition of $hol$, we have $hol(\gamma)=p \circ \tilde{\gamma}(1)$ for each element $\gamma$ in $\pi_1(N,x)$ where $p$ is the canonical projection $p \colon \tilde{N} \longrightarrow \tilde{N}/\tilde{\mathcal{G}}=G$ and $\tilde{\gamma}$ is the lift of $\gamma$ to $\tilde{N}$ such that $\tilde{\gamma}(0)=\tilde{x}$. Then we have $\overline{f}(hol(\gamma))=\overline{f} \circ p \circ \tilde{\gamma}(1)= p \circ \tilde{f} \circ \tilde{\gamma}(1) = p \circ \tilde{f \circ \gamma}(1) = hol(f \circ \gamma)$, since $\tilde{f}$ fixes $\tilde{x}$.

(3) is clear from the construction.
\end{proof}

Note that every element of $\pi_{0}(\Diff(N,\mathcal{G}))$ is represented by an element of $\Diff(N,x,\mathcal{G})$. In fact, $(N,\mathcal{F})$ is a Lie foliation and the identity component of $\Diff(N,\mathcal{G})$ acts to $N$ transitively. Hence Proposition~\ref{algebraicaction} is reduced to the algebraicity of $A \colon \Aut(G,\Gamma) \longrightarrow \Aut(E^{0,l}_2)$. 

We prove the following proposition applying Mal'cev theory.

\begin{prop}\label{malcev}
Let $G$ be an $l$-dimensional simply connected nilpotent Lie group and $\Gamma$ be a finitely generated dense subgroup of $G$. We put $\mathfrak{g}=\Lie(G)$. If we denote the group of automorphisms of $G$ which preserves $\Gamma$ by $\Aut(G,\Gamma)$, then the image of the canonical action
\begin{equation}
A \colon \Aut(G,\Gamma) \longrightarrow \Aut(H^{l}(\mathfrak{g}))
\end{equation}
is contained in the set of algebraic numbers where $\Aut(H^{l}(\mathfrak{g}))$ is canonically identified with $\mathbb{R}-\{0\}$.
\end{prop}

\begin{proof}
At first, we apply Mal'cev theory following Ghys~\cite{Ghy}. We refer to~\cite{Mal} and Chapter II of~\cite{Rag} on Mal'cev theory. We have a simply connected nilpotent Lie group $H$ and an embedding $i \colon \Gamma \longrightarrow H$ such that $i(\Gamma)$ is a uniform lattice of $H$ by Mal'cev theory. The homomorphism $i^{-1} \colon i(\Gamma) \longrightarrow G$ can be extend to a homomorphism $\pi \colon H \longrightarrow G$ again by Mal'cev theory. $\pi$ is clearly surjective, since $\Gamma$ is dense in $G$. We also have a lift $\tilde{f}$ of $f$ which is an automorphism of $H$ and preserves $i(\Gamma)$ again by Mal'cev theory, since $i(\Gamma)$ is a uniform lattice of $H$. Then we have a diagram:

\begin{equation}\label{diag2}
\xymatrix{ 
H \ar[dd]_{\pi} \ar[rrr]^{\tilde{f}} & & & H \ar[dd]^{\pi} \\
         & \Gamma \ar[ul]^{i} \ar[ld] \ar[r]^{f} & \Gamma \ar[ur]_{i} \ar[rd] \\
G \ar[rrr]^{f} & & & G }
\end{equation}
which satisfies the following conditions:
\begin{description}
\item[(a)] $H$ is nilpotent,
\item[(b)] $\pi$ is surjective and
\item[(c)] $i$ is an embedding of a uniform lattice.
\end{description}

Proposition~\ref{malcev} follows from the following lemma: 

\begin{lem}\label{ind}
If we have a diagram~\eqref{diag2} which satisfies conditions (a), (b) and (c), then $A(f) \colon H^{l}(\mathfrak{g}) \longrightarrow H^{l}(\mathfrak{g})$ is algebraic under the canonical identification of $\Aut(H^{l}(\mathfrak{g}))=\mathbb{R}-\{0\}$.
\end{lem}

We prove Lemma~\ref{ind} inductively on the rank of $H$ as a nilpotent Lie group. 

If $H$ is abelian, $(H,\Gamma)$ is isomorphic to $(\mathbb{R}^{m},\mathbb{Z}^{m})$. Then $f$ is an element of $\SL(m;\mathbb{Z})$. $f$ induces a homomorphism $\hat{f} \colon \wedge^{l} (\mathbb{R}^{m})^{*} \longrightarrow \wedge^{l} (\mathbb{R}^{m})^{*}$ and $\hat{f}$ has integral entries with respect to the standard basis, since the entries of $\hat{f}$ are minor determinants of $f$. $H^{l}(\mathfrak{g})$ is generated by an element represented by left invariant volume forms of $G$ and left invariant volume forms of $G$ are eigenvectors of $\hat{f}$ by the diagram~\eqref{diag2}. Then $A(f)$ is algebraic, since $A(f)$ is the eigenvalue of $\hat{f}$ with respect to left invariant volume forms of $G$.

Assume that we have the diagram~\eqref{diag2} which satisfies the conditions (a),(b) and (c) where the rank of $H$ is $n$ and the claim of Lemma~\ref{ind} is correct if the rank of $H$ is less than $n$. Then $A([f,f])$ and $A(H_1(f))$ are algebraic under the identification of $H^{l'}([\mathfrak{g},\mathfrak{g}])$ and $H^{l''}(\mathfrak{g}/[\mathfrak{g},\mathfrak{g}])$ with $\mathbb{R}-\{0\}$ where $l'=\dim [\mathfrak{g},\mathfrak{g}]$ and $l''=\dim \mathfrak{g}/[\mathfrak{g},\mathfrak{g}]$ by the following diagrams:

\begin{equation}
\xymatrix{ 
[H,H] \ar[dd]_{[\pi,\pi]} \ar[rrr]^{[\tilde{f},\tilde{f}]} & & & [H,H] \ar[dd]^{[\pi,\pi]} \\
         & \Gamma \cap [H,H] \ar[ul]^{[i,i]} \ar[ld] \ar[r]^{[f,f]} & \Gamma \cap [H,H] \ar[ur]_{[i,i]} \ar[rd] \\
[G,G] \ar[rrr]^{[f,f]} & & & [G,G] }
\end{equation}
and
\begin{equation}
\xymatrix{ 
H/[H,H] \ar[dd]_{H_1(\pi)} \ar[rrr]^{H_1(\tilde{f})} & & & H/[H,H] \ar[dd]^{H_1(\pi)} \\
         & p(\Gamma) \ar[ul]^{H_1(i)} \ar[ld] \ar[r]^{H_1(f)} & p(\Gamma) \ar[ur]_{H_1(i)} \ar[rd] \\
G/[G,G] \ar[rrr]^{H_1(f)} & & & G/[G,G] }
\end{equation}
where $p$ is the canonical projection $H \longrightarrow [H,H]$. Then we have the algebraicity of $A(f)$, since we have $A(f)=A([f,f])\cdot A(H_1(f))$ under the identification of $\Aut(H^{l}(\mathfrak{g})), H^{l'}([\mathfrak{g},\mathfrak{g}])$ and $H^{l''}(\mathfrak{g}/[\mathfrak{g},\mathfrak{g}])$ with $\mathbb{R}-\{0\}$ by the following diagram:
  
\begin{equation}
\xymatrix{ 0 \ar[r] & [G,G] \ar[r] \ar[d]^{[f,f]} & G \ar[r] \ar[d]^{f} & G/[G,G] \ar[r] \ar[d]^{H_1(f)} & 0 \\
 0 \ar[r] & [G,G] \ar[r] & G \ar[r] & G/[G,G] \ar[r] & 0.}
\end{equation}

Hence Lemma~\ref{ind} and Proposition~\ref{malcev} are proved.
\end{proof}

\section{Examples}

\subsection{Torus fibration over $S^1$}

Let $A$ be an element of $\SL(n;\mathbb{Z})$ with an eigenvector $v$ with respect to an eigenvalue $\lambda$. Assume that the components of $v$ are linearly independent over $\mathbb{Q}$ and $\lambda$ is a positive real number. For an example, take

\begin{equation}
A=\begin{pmatrix} 2 & 1 \\ 1 & 1 \end{pmatrix}, v= \begin{pmatrix} \frac{\sqrt{5}-1}{2} \\ -1 \end{pmatrix}, \lambda=\frac{3-\sqrt{5}}{2}. 
\end{equation}
$A$ induces a diffeomorphism $\overline{A}$ on $T^{n}=\mathbb{R}^{n}/\mathbb{Z}^{n}$. We denote the mapping torus $T^{n} \times [0,1]/(\overline{A}w,0) \sim (w,1)$ of $\overline{A}$ by $M$ and define a map $\pi \colon M \longrightarrow S^1$ by $\pi([(w,t)])=t$, which gives a $T^n$ fibration over $S^1$. Since $v$ is an eigenvector of $A$, we have a foliation $\mathcal{F}$ on $M$ formed by the lines parallel to $v$ in each $T^{n}$ fiber of $\pi$. By the assumption on $v$, the leaves of $\mathcal{F}$ are dense in the fibers of $\pi$. $(M,\mathcal{F})$ is Riemannian since we can construct a bundle-like metric $g$ of $(M,\mathcal{F})$ by $g = \rho(t)g_0 + (1-\rho(t))g_0 + dt \otimes dt$ where $g_0$ is the flat metric on $T^n$ and $\rho(t)$ is a smooth function which satisfies $\rho(t)=0$ near $0$ and $\rho(t)=1$ near $1$. Note that the structure Lie algebra of $(M,\mathcal{F})$ is abelian.

In this case, the mean curvature form of $(M,\mathcal{F},g)$ is a closed form on $S^1$ and we can calculate the \'{A}lvarez class $[\kappa_b]$ of $(M,\mathcal{F})$ directly to obtain
\begin{equation}
\int_{S^1}[\kappa_b]= \log \lambda.
\end{equation}

\subsection{A Riemannian foliation on a solvmanifold}

We present an example of a $2$-dimensional Riemannian foliation on a $6$-dimensional nilmanifold bundle over $S^1$ with nonabelian nilpotent structure Lie algebra and nontrivial \'{A}lvarez class.

Let $p$ be a prime and $\alpha$ be an element of $\mathbb{Z}(\sqrt{p})$ which has an inverse $\beta$ in $\mathbb{Z}(\sqrt{p})$. We put $k=\GCD(\alpha_2,\beta_2)$ where $\alpha_2$ and $\beta_2$ are integers which satisfies $\alpha=\alpha_1+\alpha_2\sqrt{p}$ and $\beta=\beta_1+\beta_2\sqrt{p}$ for some integers $\alpha_1$ and $\beta_1$. Let $G$ be a nilpotent Lie group which is defined by
\begin{equation}
G=\Big\{ \begin{pmatrix} 1 & x & z \\ 0 & 1 & y \\ 0 & 0 & 1 \end{pmatrix} \Big| x,y,z \in \mathbb{R} \Big\}
\end{equation}
and $\Gamma$ be a subgroup of $G$ which is generated by
\begin{equation}
A_1=\begin{pmatrix} 1 & 1 & 0 \\ 0 & 1 & 0 \\ 0 & 0 & 1 \end{pmatrix},
A_2=\begin{pmatrix} 1 & 0 & 0 \\ 0 & 1 & 1 \\ 0 & 0 & 1 \end{pmatrix},
A_3=\begin{pmatrix} 1 & k\sqrt{p} & 0 \\ 0 & 1 & 0 \\ 0 & 0 & 1 \end{pmatrix},
A_4=\begin{pmatrix} 1 & 0 & 0 \\ 0 & 1 & k\sqrt{p} \\ 0 & 0 & 1 \end{pmatrix}.
\end{equation}
$\Gamma$ is dense in $G$, since we have $[A_1,A_2]=\begin{pmatrix} 1 & 0 & 1 \\ 0 & 1 & 0 \\ 0 & 0 & 1 \end{pmatrix}$, $[A_1,A_4]=\begin{pmatrix} 1 & 0 & k\sqrt{p} \\ 0 & 1 & 0 \\ 0 & 0 & 1 \end{pmatrix}$ and the Lie algebra of the closure of $\Gamma$ is equal to the Lie algebra of $G$.

We define a Lie group $H$ by $H=G \oplus \mathbb{R}^{2}$. Let $\iota'$ be a homomorphism $\iota' \colon \Gamma \longrightarrow \mathbb{R}^{2}$ which is defined by
\begin{equation}
\iota'\Big(\begin{pmatrix} 1 & x_1 + x_2 \sqrt{p} & z \\ 0 & 1 & y_1 + y_2 \sqrt{p} \\ 0 & 0 & 1 \end{pmatrix} \Big)=(x_2,y_2)
\end{equation}
where $x_1,x_2,y_1$ and $y_2$ are integers and we define an embedding $\iota \colon \Gamma \longrightarrow H$ by $\iota(g)=(g,\iota'(g))$ for every $g$. Then $\iota(\Gamma)$ is a uniform lattice of $H$. The fibers of the first projection $H \longrightarrow G$ are preserved by the right multiplication of $\iota(\Gamma)$ to $H$ and define a $G$-Lie foliation $\mathcal{G}$ of dimension $2$ and codimension $3$ on $H/\iota(\Gamma)$.

Let $\begin{pmatrix} a' & b' \\ c' & d' \end{pmatrix}$ be an element of $\SL(2;\mathbb{Z})$ and put $\begin{pmatrix} a & b \\ c & d \end{pmatrix}=\begin{pmatrix} a'\alpha & b'\alpha \\ c'\alpha & d'\alpha \end{pmatrix}$. Let $f$ be a map $G \longrightarrow G$ defined by 
\begin{equation}
f\begin{pmatrix} 1 & x & z \\ 0 & 1 & y \\ 0 & 0 & 1 \end{pmatrix}=\begin{pmatrix} 1 & ax+by & \alpha^{2} z+acx^2+bdy^2+bcxy \\ 0 & 1 & cx+dy \\ 0 & 0 & 1 \end{pmatrix}.
\end{equation}
Then $f$ is a homomorphism from $G$ to $G$ by the definition of $\begin{pmatrix} a & b \\ c & d \end{pmatrix}$ and $\alpha$. Since $f$ is bijective, $f$ is an automorphism. Clearly $f$ preserves $\Gamma$. Moreover $f(\Gamma)=\Gamma$ follows from the definition of $k$ and $\Gamma$. Note that there exists an element $g$ of $\Gamma$ which satisfies $g=\begin{pmatrix} 1 & 0 & \delta \\ 0 & 1 & 0 \\ 0 & 0 & 1 \end{pmatrix}$ for every $\delta$ in $\mathbb{Z}(\sqrt{p})$ which has a form $\delta=\delta_1+\delta_2 k\sqrt{p}$ where $\delta_1$ and $\delta_2$ are integers.

Since $\iota(\Gamma)$ is a uniform lattice of $H$, the automorphism $f|_{\Gamma}$ of $\Gamma$ is uniquely extended to an automorphism $\tilde{f}$ of $H$ and induces a diffeomorphism $\overline{f}$ on $H/\iota(\Gamma)$. Since $\overline{f}$ preserves $\mathcal{G}$, we have a foliation $\mathcal{F}$ on the mapping torus $M$ of $\overline{f}$ which is Riemannian. By $f^{*}(dx \wedge dy \wedge dz)=\alpha^{3} dx \wedge dy \wedge dz$ and Proposition~\ref{holonomyformula}, we have 
\begin{equation}
\int_{S^1}[\kappa_b]= \log \alpha^3
\end{equation}
where $S^1$ is the base space of the canonical fibration $M \longrightarrow S^1$ of the mapping torus and $[\kappa_b]$ is the \'{A}lvarez class of $(M,\mathcal{F})$.

\end{document}